\documentclass{amsart}[12pt]

\usepackage{hyperref}
\usepackage{tikz}
\usepackage{amsmath, amssymb, amsthm, amsfonts, amsxtra, mathrsfs, mathtools, dsfont} 

\def\E{\mathbb{E}}
\def\R{\mathbb{R}}
\def\Z{\mathbb{Z}}

\def\EE{\mathcal{E}}

\def\calD{\mathcal{D}}
\def\calM{\mathcal{M}}
\def\calT{\mathcal{T}}

\def\hom{\text{Hom}}
\def\isom{\text{Isom}}
\def\so{\mathfrak{so}}

\newtheorem{innercustomgeneric}{\customgenericname}
\providecommand{\customgenericname}{}
\newcommand{\newcustomtheorem}[2]{%
  \newenvironment{#1}[1]
  {%
   \renewcommand\customgenericname{#2}%
   \renewcommand\theinnercustomgeneric{##1}%
   \innercustomgeneric
  }
  {\endinnercustomgeneric}
}

\newcustomtheorem{customthm}{Theorem}

\newtheorem{theorem}{Theorem}[section]
\newtheorem{lemma}[theorem]{Lemma}

\newtheorem{conjecture}[theorem]{Conjecture}

\title{
    Cobordism of Domes Over Curves
}

\author{
    Robert Miranda
}

\begin{document}




\begin{abstract}
    An integral curve is a closed piecewise linear curve comprised of unit intervals. A dome is a polyhedral surface whose faces are equilateral triangles and whose boundary is an integral curve. Glazyrin and Pak showed that not every integral curve can be domed by analyzing the case of unit rhombi, and conjectured that every integral curve is cobordant to a unit rhombus. We show that this is false for oriented domes, but that every integral curve is orientably cobordant to the union of finitely many unit rhombi.
\end{abstract}

\maketitle{}

\section{Introduction}

Let $\gamma$ be a closed piecewise linear curve in three dimensional Euclidean space $\E$. We say that $\gamma$ is integral if all intervals have unit length. We also consider the case where $\gamma$ has multiple connected components, in which case we say that $\gamma$ is a union of integral curves. Now let $S$ be a piecewise linear complex in $\E$ with boundary $\gamma$ and whose facets are all unit triangles. Note that $S$ need not be embedded or immersed. We say that $S$ is a \textit{dome over} $\gamma$, that $\gamma$ is \textit{spanned} by $S$, and that $\gamma$ \textit{can be domed}.\newline

As early as 2005, Kenyon asked if every integral curve can be domed, see ~\cite{Ken}. In 2022, this was shown to be false by Glazyrin and Pak, who proved a necessary condition for a unit rhombus to be domed in ~\cite{Gla}. Moreover, Glazyrin and Pak conjectured that this is in some sense the only restriction that prevents a general integral curve from being domed.

\begin{conjecture}[{\cite[Conj. ~5.14]{Gla}}]
    \label{conj}
    For every integral curve $\gamma$, there is a unit rhombus $\rho$ and a dome over ~$\gamma \cup \rho$.
\end{conjecture}

Formally, we say that two integral curves $\gamma$ and $\eta$ are \textit{cobordant} if there is a dome over $\gamma \cup \eta$. The dome will be called a \textit{cobordism} between $\gamma$ and $\eta$. When the cobordism is orientable, we say that $\gamma$ and $\eta$ are \textit{orientably cobordant}. Note that if $\gamma$ and $\eta$ are cobordant, then $\gamma$ can be domed if and only if $\eta$ can be domed, but the converse is not necessarily true. In this context, Glazyrin and Pak asked if every integral curve is cobordant to a unit rhombus. This is false for orientable cobordisms.

\begin{theorem}
    \label{thm:two-one}
    There exists an integral curve $\gamma$ of length $5$ which is not orientably cobordant to any unit rhombus $\rho$.
\end{theorem}

Our proof of this negative result is non-constructive. We show that for almost all pairs of unit rhombi $\rho_1, \rho_2$, the union $\rho_1 \cup \rho_2$ is not orientably cobordant to any third unit rhombus $\rho_3$. Then we give a cobordism between $\rho_1 \cup \rho_2$ and $\gamma$. Our work builds upon techniques introduced by Anan'in and Korshunov, who gave a second proof that Kenyon's question is false in \cite{Ana}. Their proof is also non-constructive, and shows that almost all integral curves cannot be domed. They consider a boundary map from the moduli space of domes to the moduli space of integral curves, and prove that its image has measure zero in the orientable case. Our result extends their work by allowing unions of integral curves.\newline

We also prove a weaker version of Conjecture \ref{conj}, posed by Glazyrin and Pak in ~\cite[Conj. ~5.15]{Gla}, which allows for finitely many rhombi.

\begin{theorem}
    \label{thm:finite-cobordism}
    For every integral curve $\gamma$, there is a finite set of unit rhombi $\rho_1, \dotsc, \rho_k$ such that $\gamma$ is orientably cobordant to $\rho_1 \cup \dotsc \cup \rho_k$. Moreover, for $|\gamma| = n$, it suffices to take $k = n^2 + 2n - 12$ rhombi.
\end{theorem}

Our proof is constructive, and uses \textit{rhombus equivalence} to reduce a generic integral curve to a planar integral curve. We use $B_r(v)$ to mean the sphere of radius $r > 0$ around a point $v \in \E$, and for a set of points $v_1, \dotsc, v_n \in E$, we use $[v_1 \dotsc v_n]$ to mean the integral curve through $v_1, \dotsc, v_n, v_{n+1} = v_1$. For consecutive vertices $u, v, w$ in $\gamma$, we can replace $v$ by some point $v' \in B_1(u) \cap ~B_1(w)$ by attaching the rhombus $[uvwv']$. We call this a rhombus equivalence. Then we prove the theorem directly for planar integral curves. Our approach is similar to ideas introduced in \cite[\S 2]{Gla}.

\subsection*{Outline of the paper}
We prove Theorem \ref{thm:finite-cobordism} in Section 2 because it is mostly self-contained. Then in Section 3 we introduce necessary notions to describe the moduli space of domes and curves. This is a generalization of many of the definitions and results given in \cite{Ana} to allow for domes to bound unions of integral curves. Then we prove Theorem \ref{thm:two-one} in Section 4 with these techniques. Final remarks are given in Section 5.

\subsection*{Notation}
As previously mentioned, $B_r(v)$ is the sphere of radius $r > 0$ around a point $v \in \E$, and for a list of points $v_1, \dotsc, v_n \in \E$, let $[v_1 \dotsc v_n]$ be the integral curve $\gamma$ with vertices at the given points $v_1, \dotsc, v_n, v_{n+1}$, where we use the convention that $v_{n+1} = v_1$ for an integral curve of length $n$ throughout. For two points $v, w$, let $(v, w)$ be the line containing the two points, let $[v, w]$ be the line segment connecting the two points, and let $|v, w|$ be the distance between the two points. For an integral curve $\gamma$, let $|\gamma|$ be the sum of all edge lengths of $\gamma$. Let $\calM_n$ denote the set of all integral curves of length $n$. The set $\calM_4$ is important in the proofs and is called the set of unit rhombi.

\section{Every integral curve is cobordant to finitely many unit rhombi}

In this section we prove Theorem \ref{thm:finite-cobordism}, that every integral curve $\gamma \in \calM_n$ is orientably cobordant to a finite union of rhombi $\rho_1 \cup \dotsc \cup \rho_k$.\newline

We say that two integral curves $\gamma$ and $\eta$ are \textit{rhombus equivalent} if there exist finitely many unit rhombi $\rho_1, \dotsc, \rho_k$ and a dome over $\gamma \cup \eta \cup \rho_1 \cup \dotsc \cup \rho_k$. Here we say $k$ is the number of \textit{rhombi used} in the rhombus equivalence. This is similar to the definition of \textit{flip equivalence}. (See e.g. \cite[\S ~2.4]{Gla}.) We also want to distinguish when a rhombus equivalence is orientable. However, when some of the integral curves share common edges, the resulting dome will not be purely $2$-dimensional, so we define orientability carefully. For an integral curve $\gamma$, we call the surface formed by gluing the boundary of a disk to $\gamma$ the \textit{closure} of $\gamma$. We say a dome over $\gamma, \eta, \rho_1, \dotsc, \rho_k$ is \textit{orientable} if the dome, along with the closure of $\gamma, \eta, \rho_1, \dotsc, \rho_k$ forms a closed, orientable surface. Clearly, if an integral curve $\gamma$ is orientably rhombus equivalent to an integral curve which can be domed, then $\gamma$ satisfies the existence condition of Theorem ~\ref{thm:finite-cobordism}, and we can prove the theorem by keeping track of the number of rhombi used. \newline

First, we show that every integral curve is orientably rhombus equivalent to a planar integral curve. An integral curve is \textit{planar} if it lies in a plane $H \subset \E$.

\begin{lemma}
    \label{lemma:planar}
    Every integral curve $\gamma$ is orientably rhombus equivalent to a planar integral curve. Moreover, for $|\gamma| = n$, it suffices to use $k = {n \choose 2}$ rhombi.
\end{lemma}

\begin{proof}
    Choose two vertices $v$ and $w$ of $\gamma$, and take any plane $H$ containing the line $(v, w)$. This gives a decomposition of $\gamma$ as two integral paths containing $v$ and $w$, and we show that both can be made to lie entirely in $H$ via rhombus equivalence. \newline

    Suppose $\eta$ is an integral path with vertices $v = v_1, \dotsc, v_m, v_{m+1} = w$. Let $h_i$ denote the distance from $v_i$ to the plane $H$, hence $h_1 = h_{m+1} = 0$. (See Figure 1.) Note that the heights $h_i$ are all nonnegative, but the $v_i$ may lie on different sides of $H$; this will not affect the proof. Choose $v_i$ maximizing $h_i$. If the same height is achieved by multiple ~$v_i$, choose the one with the smallest index $i$. Note that this means $h_{i-1} < h_i \geq h_{i+1}$.\newline

    \begin{figure}[!ht]
    \begin{center}
    \includegraphics{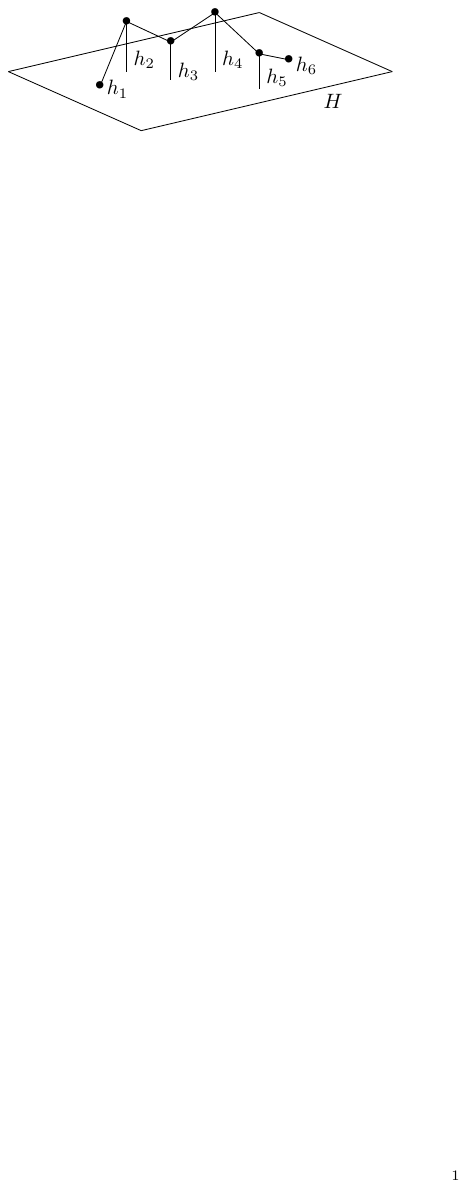}
    \vskip-.6cm \
    \caption{An integral path $\eta$ above a plane $H$ with heights $h_i$.}
    \end{center}
    \end{figure}

    Consider the circle $B_1(v_{i-1}) \cap B_1(v_{i+1})$ which contains $v_i$. For each point $v_i' \in B_1(v_{i-1}) \cap B_1(v_{i+1})$, we can replace $v_i$ by $v_i'$ in $\eta$ by rhombus equivalence via $[v_{i-1} v_i v_{i+1} v_i']$. (See Figure 2.) Choose $v_i'$ to minimize $h_i'$, the new distance from $v_i'$ to $H$. If $h_i' = 0$, then $v_i'$ is in $H$ and we are done with this vertex. If not, then we can at least guarantee that $h_i' < h_i$. Choosing $v_i'$ so that the line $(v_{i-1}, v_i')$ is parallel to the line $(v_i, v_{i+1})$, (i.e., choosing $v_i'$ so that the rhombus $[v_{i-1} v_i v_{i+1} v_i']$ is planar) we see that $h_i' \leq h_{i-1} < h_i$ because $h_{i+1} \leq h_i$. Thus we can always decrease the maximum height $h_i$ by a rhombus equivalence.\newline

    Note also that either $v_{i-1}$ or $v_{i+1}$ may be on the other side of $H$ relative to $v_i$, in which case the circle $B_1(v_{i-1}) \cap B_1(v_{i+1})$ will always intersect $H$. (If $B_1(v_{i-1}) \cap B_1(v_{i+1})$ did not intersect $H$, then $h_i$ would not be a maximum height.) So in this case, we can always choose $v_i' \in H$.
    \begin{figure}[!ht]
    \begin{center}
    \includegraphics{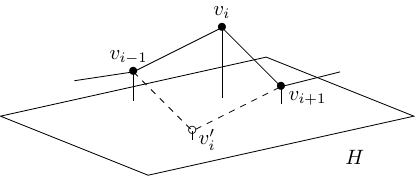}
    \vskip-.6cm \
    \caption{Adding a rhombus $[v_{i-1} v_i v_{i+1} v_i']$.}
    \end{center}
    \end{figure}

    To prove our result, we view rhombus equivalence as an operation on the sequence $0, h_2, \dotsc, h_m, 0$ which, in the worst case, takes the first maximum element $h_i$ and replaces it with the smaller value of $h_{i-1}$. Because $h_1 = 0$, after repeated application, rhombus equivalence will eventually reduce the height sequence to all $0$s. If the maximum height is achieved $k$ separate times, then it will take $k$ rhombus equivalences to reduce the maximum height of the sequence to a smaller value. Thus the worst case is when $h_1, \dotsc, h_m$ are all distinct, when ${m \choose 2}$ flips are needed. So for the whole integral curve $\gamma$, we will need at most ${m \choose 2} + {n - m \choose 2} \leq {n \choose 2}$ rhombi to find an integral curve which lies entirely in $H$.\newline

    To see that the rhombus equivalence is orientable, note that the closure of all the rhombi used in the rhombus equivalence is homeomorphic to an annulus. Thus the cobordism over $\gamma$, all of the rhombi used, and the resulting planar integral curve is orientable.
\end{proof}

We say that an integral curve $\gamma = [v_1 \dotsc v_n]$ is \textit{$\epsilon$-packing around} $v_i$ if for all $j$, the length $|v_i, v_j| < \epsilon$. In this case, we say that $v_i$ is the \textit{packing center} of $\gamma$, and when the packing center is understood we just say that $\gamma$ is \textit{$\epsilon$-packing}. Next we prove that every planar integral curve is orientably rhombus equivalent to a $2$-packing planar integral curve.

\begin{lemma}
    \label{lemma-packing}
    Every planar integral curve $\gamma$ is orientably rhombus equivalent to a planar integral curve which is $2$-packing around one of its vertices. Moreover, if $|\gamma| = n$, then it suffices to use $k = {n \choose 2}$ rhombi.
\end{lemma}

\begin{proof}
    The existence of a rhombus equivalent $2$-packing curve follows from the Steinitz Lemma, which can be stated as follows. (See e.g. ~\cite[Theorem ~2.1]{Bar}.) For each dimension $d > 0$, there is a constant $\beta_d$ such that for any unit vectors $u_1, \dotsc, u_n \in \R^d$ satisfying $u_1 + \dotsc + u_n = 0$, there is a permutation $\sigma \in S_n$ such that for each $1 \leq i \leq n$,
    \[
        |u_{\sigma(1)} + \dotsc + u_{\sigma(i)}| \leq \beta_d
    \]
    For $d = 2$, Bergstr\"{o}m showed the optimal value is $\beta_2 = \sqrt{5}/2$ in \cite{Ber}. However, we use the constant $2$ because it is sufficient for the proof of Theorem \ref{thm:finite-cobordism} and gives clearer exposition.\newline

    Viewing the integral curve $\gamma = [v_1 \dotsc v_n]$ as a set of vectors $u_1, \dotsc, u_n$, with $u_i$ pointing from $v_i$ to $v_{i+1}$, we can swap a pair of consecutive vectors $u_i$ and $u_{i+1}$. We add a rhombus containing the endpoints $v_{i-1}, v_i, v_{i+1}$, and the point $v_i$ reflected across the line $(v_{i-1}, v_{i+1})$. This corresponds to the simple transposition $(i, i+1) \in S_n$. And because we can achieve any permutation $\sigma \in S_n$ by the product of simple transpositions, there exists a rhombus equivalent $2$-packing planar integral curve around $v_1$. (See e.g. \cite{Sta}.) Moreover, for every permutation $\sigma \in S_n$, the maximal length of $\sigma$ in terms of simple transpositions is ${n \choose 2}$, so this number of rhombi is sufficient.\newline

    Again, we see that this rhombus equivalence is orientable because the closure of all of the rhombi used in the rhombus equivalence is homeomorphic to an annulus. 
\end{proof}

Now we prove the result of Theorem \ref{thm:finite-cobordism} for planar integral curves. Our proof is direct when $|\gamma| = 5$, and when $|\gamma| > 5$ we divide $\gamma$ into pentagons.

\begin{customthm}{1.3}
    \label{thm:finite-cobordism}
    Every integral curve $\gamma$ is orientably cobordant to $\rho_1 \cup \dotsc \cup \rho_k$ for a finite set of unit rhombi $\rho_1, \dotsc, \rho_k \in \calM_4$. Moreover, for $|\gamma| = n$, it suffices to take $k = n^2 + 2n - 12$ rhombi.
\end{customthm}

\begin{proof}
    From Lemma \ref{lemma:planar} and Lemma \ref{lemma-packing}, we may assume without loss of generality that $\gamma$ is planar and $2$-packing around one of its vertices $v_1$. This uses $n(n-1)$ rhombi. Now we prove the statement for planar, $2$-packing integral curves by dividing them into pentagons.\newline

    First, suppose $\gamma$ is a pentagon $[v_1 \dotsc v_5]$. If the circumradius of the triangle $[v_1 v_3 v_4]$ is less than $1$, then there is a point $z \in B_1(v_1) \cap B_1(v_3) \cap B_1(v_4)$. (See Figure 3, left.) Adding the edges $[z,v_1], [z,v_3]$ and $[z,v_4]$ gives a cobordism between $\gamma$ and two unit rhombi $[v_1 v_2 v_3 z]$ and $[v_1 z v_4 v_5]$. If the circumradius of the triangle $[v_1 v_3 v_4]$ is greater than $1$, then we can perform a rhombus equivalence to construct a new integral curve $\gamma'$ with the desired circumradius. Note that at most $3$ rhombi are used.\newline

    \begin{figure}[!ht]
    \begin{center}
    \includegraphics{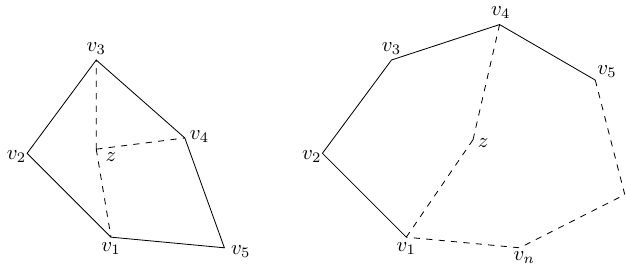}
    \vskip-.6cm \
    \caption{The $n = 5$ case and the general case.}
    \end{center}
    \end{figure}

   In the general case, for a planar, $2$-packing integral curve $\gamma = [v_1 \dotsc v_n]$ with $n > 5$, there is some point $z \in B_1(v_1) \cap B_1(v_4) \cap H$, where $H$ is the plane containing $\gamma$. (See Figure 3, right.) We can divide $\gamma$ into the planar pentagon $[v_1 v_2 v_3 v_4 z]$ and $\gamma' := [v_1 z v_4 \dotsc v_n]$, which is a planar, $2$-packing integral curve of length $n-1$. Continuing this procedure divides $\gamma$ into $n-4$ pentagons, each of which use at most $3$ rhombi. So in total at most $n(n-1) + 3(n-4) = n^2 + 2n - 12$ rhombi are used.\newline

   To see that this cobordism is orientable, note that when $n = 5$ the surface formed from the triangle and the closure of the rhombi used is homeomorphic to a disk. And in the general case, we just divide $\gamma$ into the union of pentagons.
\end{proof}

\section{Spaces of Polygons and Polyhedra}
This section introduces definitions and proves technical results we need to prove Theorem \ref{thm:two-one}, that not all integral curves are orientably cobordant to a unit rhombus. Our approach is inspired by techniques introduced by Anan'in and Korshunov, who gave an alternate proof of the negative answer to Kenyon's question for almost all integral curves in \cite{Ana}. Specifically, we generalize Anan'in and Korshunov's results to graph surfaces with multiple boundary components corresponding to the union of multiple integral curves. We state our definitions here. For the original definitions and for a more detailed introduction, see \cite[\S 2, \S 3]{Ana}.\newline

A \textit{sample polygon} is a finite $1$-complex $P = (U, F)$ whose underlying space is homeomorphic to $S^1$. Here $U$ is a set of vertices and $F = \{f_1, \dotsc, f_k\}$ is a set of edges oriented and cyclically ordered with respect to the orientation of the circle. $F$ is equipped with an edge length function $\ell : F \to \R_{>0}$ which satisfies a nondegeneracy condition $2 \ell(f_i) < \ell(f_1) + \dotsc + \ell(f_n)$ for all $1 \leq i \leq n$.\newline

For a sample polygon $P$, the \textit{space of polygons} $\E^P$ is the set of all continuous maps $p : P \to \E$ that are isometries on edges. That is, $\E^P$ is the set of all realizations of $P$ in Euclidean space. Let $\isom^+ \E$ be the group of orientation preserving isometries of $\E$. The \textit{moduli space of polygons} $\E^P / \isom^+ \E$ is the set of all realizations up to rigid motions. Let $\EE$ be the subgroup of $\isom^+ \E$ of translations, and $SO(3)$ be the subgroup of rotations. Note that $\isom^+ \E \cong \EE \rtimes SO(3)$. The \textit{scheme of polygons} $\E^P / \EE$ is the set of all realizations up to translations, but not rotations. Ultimately, we need to consider $\E^P / \isom^+ \E$, but in our proof we first consider $\E^P / \EE$ and then quotient by the action of $SO(3)$.\newline

For a vector of polygons $\mathbf{P} = (P_1, \dotsc, P_n)$, let $\E^{\textbf{P}}$ refer to the product $\E^{P_1} \times \dotsc \times \E^{P_n}$. Similarly, let $\E^{\textbf{P}}/\isom^+ \E$ and $\E^{\textbf{P}}/\EE$ refer to the products $\E^{P_1} / \isom^+ \E \times \dotsc \times \E^{P_n} / \isom^+ \E$ and $\E^{P_1}/\EE \times \dotsc \times \E^{P_n}/\EE$, respectively. As an abuse of terminology, we still refer to these as the \textit{space of polygons}, \textit{moduli space of polygons} and \textit{scheme of polygons} respectively. Note that the group of isometries always acts on each sample polygon separately, see $\S$ \ref{sub-isom}.  To refer to the polygons explicity, we may write $\E^{P_1, \dotsc, P_n}$ rather than $\E^{\mathbf{P}}$. (See Lemma ~\ref{lemma:two-one}.) We refer to the sets of edges as $F_1, \dotsc, F_n$, individual edges as $f^i_j \in F_i$, and the edge length functions as $\ell_i : F_i \to \R_{>0}$.\newline

As the name suggests, $\E^{\textbf{P}}/\EE$ is indeed a scheme. Note, however, that it may contain singular points. For one polygon, $p \in \E^P / \EE$ is singular if and only if all $p(f_i)$ are parallel. For multiple polygons, $\textbf{p} \in \E^{\textbf{P}}/\EE$ is singular if and only if some $p_i \in \E^{P_i}/\EE$ is singular.\newline

\begin{lemma}
    \label{lemma:singular-points}
    Let $\textbf{p} \in \E^{\textbf{P}}/\EE$ be a singular point. Then $\dim T_{\textbf{p}}(\E^{\textbf{P}}/\EE) = m + \dim(\E^{\textbf{P}}/\EE)$, where $m$ is the number of $p_i$ which are singular in $\E^{P_i}/\EE$.
\end{lemma}

\begin{proof}
    Since $\E^{\textbf{P}}/\EE = \E^{P_1}/\EE \times \dotsc \times \E^{P_n}/\EE$, we have
    \[
        \dim T_{\textbf{p}}(\E^{\textbf{P}}/\EE) = \dim T_{p_1} (\E^{P_1}/\EE) + \dotsc + \dim T_{p_n} (\E^{P_n}/\EE).
    \]
    Now for a singular $p_i$, we have $\dim T_{p_i} (\E^{P_i}/\EE) = 1 + \dim(\E^{P_i}/\EE)$ from \cite[Lemma 3.2]{Ana}, else for a smooth $p_i$, we have $\dim T_{p_i} (\E^{P_i} / \EE) = \dim(\E^{P_i} / \EE)$. Combining this for all $p_i$ in $\textbf{P}$ gives us the result.
\end{proof}

Consider the following explicit description of the scheme of polygons $\E^{\textbf{P}}/\EE$ for a list of polygons $\textbf{P}$. A point $\textbf{p} \in \E^{\textbf{P}}/\EE$ is a collection of maps $p_i : F_i \to \E$ which satisfy $\langle p_i(f^i_j), p_i(f^i_j) \rangle = \ell_i(f^i_j)^2$ for all $i, j$ and $\sum_j p_i(f^i_j) = 0$ for all $i$. Similarly, for $\textbf{p} \in \E^{\textbf{P}}/\EE$, a point $\textbf{t}$ in the tangent space $T_{\textbf{p}}(\E^{\textbf{P}}/\EE)$ is a collection of maps $t_i : F_i \to \E$ satisfying $\langle t_i(f^i_j), p_i(f^i_j) \rangle = 0$ for all $i, j$ and $\sum_j t_i(f^i_j) = 0$ for all $i$.\newline

We generalize the definition of a graph surface given in \cite{Ana} to allow for multiple boundary components as follows. A \textit{genus ~$n$ graph surface} $S$ is a finite $2$-dimensional simplicial complex with nondegenerate triangles and edges contained in a closed surface $\widehat{S}$ such that the complement $\calD := \widehat{S} \smallsetminus S$ is homeomorphic to the disjoint union of $n$ disks. This has the following data: a set of vertices $V$, a set of edges $E$ with orientation, a set of triangles $\calT$, and a map $\Phi : E \to E$ defined as $\Phi : e \mapsto -e$ which reverses orientation.\newline

A graph surface of genus $n$ has boundary $\partial \calD = G_1 \cup \dotsc \cup G_n$, where each $G_i$ is a \textit{boundary component} homeomorphic to $S^1$. Each boundary component $G_i$ can be decomposed into the union $g^i_1 \cup \dotsc \cup g^i_{k_i}$ with a cyclic order of the edges $g^i_1, \dotsc, g^i_{k_i} \in E$ such that $g^i_j$ and $g^i_{j+1}$ are consecutive for all $i, j$. Note that the list $g^i_1, \dotsc, g^i_{k_i}$ admits repetitions with the same or opposite orientation. As an important special case, note that when a graph surface $S \subset \widehat{S}$ is contained in an orientable surface and contains no triangles, then every edge appears exactly twice, once with each orientation, in the boundary $\partial \calD$. See \cite[~Remark~2.5]{Ana}, and note that each boundary component can be given the same orientation with respect to the oriented surface $\widehat{S}$.\newline

Similarly, we generalize the definition of a \textit{sample polyhedron} to allow for a genus $n$ graph surface instead of just a genus $1$ graph surface as in \cite[Definition 2.1]{Ana}. For us, a sample polyhedron will mean a genus $n$ graph surface $S$ equipped with an edge length function $\ell : \E \to \R_{>0}$ satisfying $\ell(e) = \ell(-e)$ for any edge $e \in E$ and $\ell(e_1) + \ell(e_2) > \ell(e_3)$ for any triangle $T \in \calT$ with boundary $\partial T = e_1 + e_2 + e_3$. For a given sample polyhedron $S$, we also consider the \textit{space of polyhedra} $\E^S$, the \textit{moduli space of polyhedra} $\E^S / \isom^+ \E$ and the \textit{scheme of polyhedra} $\E^S / \EE$.\newline

An explicit description for the scheme of polyhedra $\E^S / \EE$ is given by the set of continuous maps $q : E \to \E$ satisfying

\begin{enumerate}
    \item [(i)] $\langle q(e), q(e) \rangle = \ell(e)^2$ for all $e \in E$,
    \item [(ii)] $q(-e) = -q(e)$ for all $e \in E$,
    \item [(iii)] $q(e_1) + q(e_2) + q(e_3) = 0$ for all triangles $T \in \calT$ with boundary $\partial T = e_1 + e_2 + e_3$,
    \item [(iv)] for a set of generators $H \subset H_1(S, \Z)$, $\displaystyle \sum_{e \in E} h_e q(e) = 0$ for any generator $\displaystyle \sum_{e \in E} h_e e \in H$.
\end{enumerate}

Similarly, an explicit description for the tangent space $T_q(\E^S / \EE)$ is given by the set of continuous maps $s : E \to \E$ satisfying

\begin{enumerate}
    \item [(i')] $\langle s(e), q(e) \rangle = 0$ for all $e \in E$,
\end{enumerate}

and the previous conditions (ii) - (iv). See \cite[\S\S 2.7-9]{Ana}.\newline

There is a boundary map which connects sample polyhedra and sample polygons. Let $S$ be a sample polyhedron, with boundary components $G_1, \dotsc, G_n$ where $G_i = g^i_1 \cup \dotsc \cup g^i_{k_i}$ in cyclic order. For each $i, j$, define an oriented edge $f^i_j$ of length $\ell(g^i_j)$ and glue the edges $f^i_1, \dotsc, f^i_{k_i}$ into a sample polygon $P_i$. We call the polygons $\textbf{P}$ the \textit{boundary polygons} of $S$. This defines a map $\delta : F_1 \cup \dotsc \cup F_n \to E$ via $\delta : f^i_j \mapsto g^i_j$, and extends to a continuous map $\overline{\delta} : P_1 \cup \dotsc \cup P_n \to S$. Thus $\overline{\delta}(f^i_j) = g^i_j$ for all $i, j$ and $\overline{\delta}$ is an isometry on edges. Now $\overline{\delta}(P_i) = G_i$, and $\overline{\delta}(P_1 \cup \dotsc \cup P_n) = \partial \calD$. We call the map $\overline{\delta}$, or just $\delta$, the \textit{combinatorial boundary map} of the sample polyhedron $S$.\newline

The boundary map induces a continuous map on schemes $\delta : \E^S / \EE \to \E^{\textbf{P}}/\EE$ and a derivative map on the tangent spaces $d \delta : T_q (\E^S / \EE) \to T_{q \circ \delta} (\E^{\textbf{P}}/\EE)$ for any point $q : E \to \E$ in the scheme of polyhedra $\E^S / \EE$. These maps are from precomposition with $\delta$, see \cite[\S 2.10]{Ana}.\newline

Lastly, we define the collapse of a genus $n$ graph surface $S$. Let $T \in \calT$ be a triangle whose boundary $\partial T$ contains a boundary edge $g^i_j$. We may collapse $S$ at $T$. The resulting simplicial complex $S' \subset \widehat{S}$ has the same vertices $V' = V$, one less pair of oriented edges $E' = E \smallsetminus \{g^i_j, -g^i_j\}$, and one less triangle $\calT' = \calT \smallsetminus \{T\}$. Note that $S'$ is again a genus $n$ graph surface. The boundary component $G_i$ changes from $g^i_1 \cup \dotsc \cup g^i_j \cup \dotsc \cup g^i_{k_i}$ to $g^i_1 \cup \dotsc \cup -e' \cup -e \cup \dotsc \cup g^i_{k_i}$, where $\partial T = g^i_j + e + e'$. All other boundary components are unchanged. In terms of a realization $q : E \to \E$ in the scheme of polyhedra $\E^S / \EE$, collapse corresponds to restriction.

\begin{lemma}
    \label{lemma:collapse-restriction}
    Let $S'$ be a sample polyhedra obtained from $S$ by collapse of a triangle, and let $q : E \to \E$ be a realization in the scheme of polyhedra $\E^{S}/\EE$. Denote by $q' = q|_{E'} : E' \to E$ the restriction of $q$ to $E'$. Then $q' \in \E^{S'}/\EE$. Similarly, for $s : E \to \E$ in $T_q(\E^S/\EE)$, denote by $s' = s|_{E'} : E' \to \E$ the restriction of $s$ to $E'$. Then $s' \in T_{q'}(\E^{S'}/\EE)$.
\end{lemma}

\begin{proof}
    This is exactly the statement of \cite[Proposition 4.2 (i)]{Ana}. The same proof follows exactly when we allow for a genus $n$ graph surface.
\end{proof}

Consider the group $SO(3)^{\times n}$ acting on $\E^{\textbf{P}}/\EE$, where each copy of $SO(3)$ acts on the corresponding scheme of polygons $\E^{P_i}/\EE$. For $\textbf{p} \in \E^{\textbf{P}}/\EE$, we need to describe the tangent space $T_{\textbf{p}}SO(3)^{\times n} \textbf{p}$. Recall the Lie algebra $\so_3$ of the Lie group $SO(3)$ with the following description.
\[
    \so_3 = \{a \in \hom_{\R}(\E, \E) \ | \ \langle a(e), e' \rangle + \langle e, a(e') \rangle = 0 \ \forall e, e' \in E \}.
\]
Let $\so_3^{\times n}$ be the set of vectors $\textbf{a} = (a_1, \dotsc, a_n)$, where each $a_i \in \so_3$. We have the following description of the tangent space to the $SO(3)^{\times n}$ orbit on $\E^{\textbf{P}}/\EE$.

\begin{lemma}
    \label{lemma:so(3)n-orbit}
    Let $\textbf{p} \in \E^{\textbf{P}}/\EE$. The tangent space $T_{\textbf{p}}SO(3)^{\times n} \textbf{p}$ to the $SO(3)^{\times n}$-orbit of $\textbf{p}$ is given by
    \begin{align*}
        T_{\textbf{p}}SO(3)^{\times n} \textbf{p} &= T_{p_1} SO(3)p_1 \times \dotsc \times T_{p_n} SO(3) p_n \\
        &= \{\textbf{a} \circ \textbf{p} = \{a_i \circ p_i : F_i \to \E \} \ | \ \textbf{a} \in \so_3^{\times n}\}.
    \end{align*}
\end{lemma}

\begin{proof}
    The first line is immediate because each copy of $SO(3)$ acts independently on each copy of $\E^{P_i}/\EE$. Then for a single $SO(3)$ acting on $p_i \in \E^{P_i} /\EE$, we know $T_{p_i} SO(3)p_i = \{a_i \circ p_i : F_i \to \E \ | \ a_i \in \so_3\}$ from \cite[\S 3.1]{Ana}.
\end{proof}

The final ingredient is a symplectic form on the scheme of polygons $\E^{\textbf{P}}/\EE$. Consider the skew symmetric form $\omega$ given by the formula
\[
    \omega_{\textbf{p}}(\textbf{t}, \textbf{t'}) := \sum_{i=1}^n \sum_{j=1}^{k_i} \frac{t_i(f^i_j) \wedge t'_i(f^i_j) \wedge p_i(f^i_j)}{\ell_i(f^i_j)^2 \nu} = \sum_{i=1}^n \omega_{p_i}(t_i, t_i').
\]
Here $\nu$ is the volume form on $\E$, and $\textbf{t}, \textbf{t'} \in T_{\textbf{p}}(\E^{\textbf{P}}/\EE)$. $\omega_{p_i}$ refers to the sum containing all components of $P_i$, and corresponds to \cite[\S 3.3, Formula V]{Ana}.  We show that the kernel of this form is exactly the tangent space to the $SO(3)^{\times n}$ orbit of a point $\textbf{p} \in \E^{\textbf{P}}/\EE$, so that $\omega_{\textbf{p}}$ descends to the moduli space of polygons $\E^{\textbf{P}}/\isom^+ \E$. Note that by kernel we are considering $\omega_{\textbf{p}}$ as a map $T_{\textbf{p}}(\E^{\textbf{P}}/\EE) \to T_{\textbf{p}}(\E^{\textbf{P}}/\EE)^{\ast}$. That is, the kernel is the set of points $\textbf{t} \in T_{\textbf{p}}(\E^{\textbf{P}}/\EE)$ such that for all $\textbf{t'} \in T_{\textbf{p}}(\E^{\textbf{P}}/\EE)$, we have $\omega_{\textbf{p}}(\textbf{t},\textbf{t'}) = 0$. For more background on the symplectic structure of the moduli space of polygons, see \cite{KM} and \cite{Kly}.

\begin{lemma}
    \label{lemma-tangent-space}
    The tangent space $T_{\textbf{p}}(SO(3)^{\times n} \textbf{p})$ to the $SO(3)^{\times n}$ orbit of any point $\textbf{p} \in \E^{\textbf{P}}/\EE$ coincides with the kernel of the form $\omega_{\textbf{p}}$ on $T_{\textbf{p}}(\E^{\textbf{P}}/\EE)$.
\end{lemma}

\begin{proof}
    For a single $p_i \in \E^{P_i}/\EE$, we know the kernel of the form $\omega_{p_i}$ on $T_{p_i}(\E^{P_i}/\EE)$ corresponds to $T_{p_i} SO(3)p_i$ from \cite[Lemma ~3.4]{Ana}. This immediately shows that $T_{\textbf{p}}SO(3)^{\times n} \textbf{p} \subset \ker \omega_{\textbf{p}}$, as each term in the sum will be $0$. For the reverse inclusion, we repeat the dimension counting argument from \cite[~Lemma~3.4]{Ana} with our Lemma \ref{lemma:singular-points} in place of Lemma \cite[~Lemma~3.2]{Ana}. This shows the spaces have the same dimension, so they are equal.
\end{proof}

\section{Not Every Integral Curve is Cobordant to a Unit Rhombus}
This section proves that not every integral curve is orientably cobordant to a unit rhombus. Our method is to generalize \cite[Theorem 4.3]{Ana} to allow for graph surfaces of arbitrary genus. We rephrase a cobordism of a certain combinatorial type as a sample polyhedron with multiple boundary components, and project down to the smaller moduli space with one fewer boundary component to account for any possible unit rhombus to be chosen for the cobordism. 

\begin{theorem}
    \label{thm:isotropic}
    Let $S \subset \widehat{S}$ be a genus $n$ graph surface, with boundary polygons $\textbf{P}$, where $\widehat{S}$ is a closed orientable surface. Then
    \[
        \delta : \E^S / \isom^+ \E \to \E^{\textbf{P}}/\isom^+ \E,
    \]
    is isotropic. In particular, the rank of $d \delta$ is at most half the dimension of the moduli space of polygons.
\end{theorem}

Note that the skew symmetric form on $\E^{\textbf{P}}/\isom^+ \E$ is induced by the skew symmetric form $\omega$ on $\E^{\textbf{P}}/\EE$ after taking the quotient by the action of $SO(3)^{\times n}$. We show that the pullback of this skew symmetric form $\omega'$ is null. Also note that the orientability hypothesis is necessary here, see $\S$ \ref{sub-orientable}

\begin{proof}
    The proof of \cite[Theorem 4.3]{Ana} follows word for word to show that $\delta : \E^S / \EE \to \E^{\textbf{P}}/\EE$ is isotropic by a series of collapses. To show that this survives in the quotient by $SO(3)^{\times n}$, we use Lemma~\ref{lemma-tangent-space}. Because $T_{\textbf{p}} SO(3)^{\times n} \textbf{p} = \ker \omega_{\textbf{p}}$, taking the quotient by $SO(3)^{\times n}$ on $\E^{\textbf{P}}/\EE$ gives a new nondegenerate skew-symmetric form $\omega'$ on $\E^{\textbf{P}}/\isom^+ \E$, and the pullback of the form to $\E^S / \isom^+ \E$ is still null.
\end{proof}

Now we can phrase the condition of being cobordant in terms of a sample polyhedron with multiple boundary components.

\begin{lemma}
    \label{lemma:two-one}
    Let $S \subset \widehat{S}$ be a sample polyhedron with boundary polygons $P_1, P_2, P_3 \in \calM_n$, where $\widehat{S}$ is an orientable closed surface. Let $\pi$ be the projection from $\E^{P_1, P_2, P_3}/\isom^+ \E \to \E^{P_1,P_2}/\isom^+\E$. Then $\pi \circ \delta(\E^S / \isom^+ \E) \subset \E^{P_1,P_2}/\isom^+\E$ has measure $0$.
\end{lemma}

\begin{proof}
    Take the reduction of the scheme $\E^S/\isom^+\E$ because we are only interested in the set theoretical image of the space of polyhedra up to sets of measure $0$. Consider its smooth locus which is open and of full measure. Now $\delta$ is a smooth map of smooth manifolds, and by Theorem \ref{thm:isotropic} the rank of its differential is at most half the dimension of the target manifold. Let $\dim \E^{P_i}/\isom^+ \E = m$. Then $\dim \E^{P_1, P_2, P_3}/\isom^+ \E = 3m$, so $d \delta$ has rank $\leq \frac{3}{2}m$. But $\dim \E^{P_1, P_2}/\isom^+ \E = 2m$, and so the map $d\pi \circ d\delta$ has rank $\leq \frac{3}{2}m < 2m$. By Sard's theorem, the image $\pi \circ \delta(\E^S/\isom^+\E)$ has measure $0$ in $\E^{P_1,P_2}/\isom^+ \E$.
\end{proof}

We prove the main result of the section by showing that in general we cannot find a cobordism from two unit rhombi down to one unit rhombus.

\begin{theorem}
    \label{thm:two-one}
    The set of unit rhombi $\rho_1, \rho_2 \in \calM_4$ such that $\rho_1 \cup \rho_2$ is orientably cobordant to a third unit rhombus $\rho_3 \in \calM_4$ has measure $0$.
\end{theorem}

\begin{proof}
    Any dome that forms a cobordism between two unit rhombi $\rho_1, \rho_2 \in \calM_4$ and another unit rhombus $\rho_3 \in \calM_4$ is exactly a genus $3$ graph surface $S$ whose boundary $\delta(\E^S /\isom^+\E) \subset \E^{\rho_1, \rho_2, \rho_3}/\isom^+\E$. Since we can take any unit rhombus $\rho_3$, we care about the projection $\pi \circ \delta(\E^S/\isom^+\E) \subset \E^{\rho_1, \rho_2}/\isom^+\E$. This has measure $0$ from Lemma \ref{lemma:two-one}, and there are only countably many domes, so we conclude that the set of pairs of rhombi $\rho_1, \rho_2 \in \calM_4$ such that $\rho_1 \cup \rho_2$ is cobordant to a third unit rhombus $\rho_3 \in \calM_4$ has measure $0$.
\end{proof}

Glazyrin and Pak did not consider unions of integral curves, so Theorem \ref{thm:two-one} is not an answer to Conjecture \ref{conj} in the orientable case, but we give one in the following.

\begin{customthm}{1.2}
    \label{cor:one-connected-comp}
    There exists an integral curve $\gamma$ of length $5$ that is not orientably cobordant to any unit rhombus $\rho$.
\end{customthm}

\begin{proof}
    Choose a pair of unit rhombi $\rho, \rho' \in \calM_4$ such that $\rho \cup \rho'$ is not cobordant to any third unit rhombus $\rho'' \in \calM_4$. Let $\rho = [v_1 v_2 v_3 v_4]$ and $\rho' = [v_1' v_2' v_3' v_4']$. Position $\rho$ and $\rho'$ so that $v_1 = v_1'$, $v_2 = v_2'$, and the distance $|v_3, v_3'| = 1$. Note that from our definition of $\E^{\rho, \rho'}/\isom^+\E$, it does not matter how $\rho$ and $\rho'$ are oriented relative to each other, see \S\ref{sub-isom}. By adding the unit triangle $[v_2, v_3, v_3']$, we see tht $\rho \cup \rho'$ is cobordant to the perimeter $\gamma = [v_1 v_4 v_3 v_3' v_4']$. Therefore, our choice of $\rho, \rho'$ implies that $\gamma$ is not cobordant to any unit rhombus, and $\gamma$ has only one connected component.
\end{proof}

Glazyrin and Pak conjectured that a positive answer to Conjecture \ref{conj} would involve showing a cobordism between an integral curve and finitely many unit rhombi, and then a cobordism between two unit rhombi and one unit rhombus, see \cite[Conj. 5.15, Conj. 5.16]{Gla}. Thus Theorem \ref{thm:two-one} is a negative answer to \cite[Conj. 5.16]{Gla} and Theorem \ref{cor:one-connected-comp} is a negative answer to Conjecture \ref{conj}, both in the orientable case.

\section{Final remarks}

\subsection{}
\label{sub-isom}
The proof of Theorem \ref{thm:two-one} does not account for the fact that the unit rhombi in question can be oriented in some manner with respect to each other. That is, we have chosen to define $\E^{\textbf{P}}/\isom^+ \E$ with the group of orientation preserving isometries $\isom^+ \E$ acting on each polygon separately. However, we could consider $\isom^+ \E$ as acting on all of the spaces together, so instead of $\E^{P_1}/\isom^+ \E \times \dotsc \times \E^{P_n} / \isom^+ \E$ we would consider $\E^{P_1}/\isom^+ \E \times \E^{P_2} \times \dotsc \times \E^{P_n}$. This approach could answer more general questions. For example, Glazyrin and Pak conjectured that there are two unit triangles $\Delta_1, \Delta_2 \subset \E$ which are not cobordant, see \cite[Conjecture 5.13]{Gla}. In this question, the relative translation and rotation of the triangles is important, so our current techniques are insufficient.\newline

Additionally, consider Steinhaus' 1957 problem on tetrahedral chains, see \cite{Stei}. A tetrahedral chain is a polyhedron constructed by attaching regular tetrahedra along faces to form a chain. These can be viewed as cobordisms between two triangles that satisfy stricter conditions than general domes. Steinhaus asked if tetrahedral chains can be closed, and if they are dense in $\E$. In contrast to general domes, the first question was shown to be false by Swierczkowski in \cite{Swi}. However, the second question remains open. Recently, Stewart showed in \cite{Stew} that the group generated by reflections across the faces of a regular tetrahedra is dense in $SO(3)$, but this does not resolve Steinhaus' question because it does not consider the translations in the full group $SO(3) \rtimes \E$. I.e., Stewart's approach considers only the relative rotation, not the relative translation of the two triangles.

\subsection{}
\label{sub-min}
The minimal number $k$ of unit rhombi needed for the cobordism in Theorem $\ref{thm:finite-cobordism}$ gives a measure of complexity for an integral curve $\gamma$. In contrast to Conjecture \ref{conj} which proposed that $k = 1$ for all $\gamma$, we conjecture that $k = \Theta(|\gamma|)$. Theorem \ref{thm:finite-cobordism} gives a quadratic upper bound $k = O(|\gamma|^2)$. We do not believe this is optimal. For instance, the proofs of Lemma \ref{lemma:planar}, Lemma \ref{lemma-packing} do not use any triangles in the rhombus equivalences; only rhombi are added. And the proofs could be optimized to improve the coefficients of the bound, for instance, by choosing three points in the plane $H$ rather than two in the proof of Lemma \ref{lemma:planar}, but this would still give a quadratic bound in $|\gamma|$. Additionally, Pak proposes a modified proof of Theorem \ref{thm:finite-cobordism} which uses a reduction to generic integral curves rather than planar integral curves to improve to a linear bound.\footnote{I. Pak, Personal communication.} (See \cite[\S 2]{Gla} for terminology.) We also conjecture that the construction in Corollary \ref{cor:one-connected-comp} can be extended to arbitrarily many rhombi to prove that $k = \Omega(|\gamma|)$. 

\subsection{}
\label{sub-orientable}
The orientability hypothesis in Theorem \ref{thm:isotropic} is necessary for the cancellation argument, see \cite[~Remark~2.3]{Ana} for more detail, and \cite[Remark~4.4]{Ana} for a counterexample with a non-orientable surface. This limits the techniques used in this paper to the orientable case, and new techniques may be necessary to study non-orientable cobordisms.

\subsection*{Acknowledgements}
I am grateful to Alexey Glazyrin and Dmitrii Korshunov, for helpful comments and clarifications. I am also thankful to Igor Pak for introducing me to this problem and for careful reading of my drafts. Finally, I would like to thank multiple anonymous referees for insightful comments and suggestions.


\begin{thebibliography}{abcdefg}

\bibitem[AK24]{Ana}
S.~Anan'in, D.~Korshunov,
Moduli spaces of polygons and deformations of polyhedra with boundary,
\emph{Geom. Dedicate}~\textbf{218} (2024), 1-19.

\bibitem[Bar08]{Bar}
I.~B\'ar\'any,  On the power of linear dependencies, in \emph{Building bridges},
Springer, Berlin, 2008, 31--45.

\bibitem[Ber31]{Ber}
V.~Bergstr\"om, Zwei S\"atze \"uber ebene Vectorpolygone (in German), \emph{Abh.\
Math.\ Sem.\ Univ.\ Hamburg}~\textbf{8} (1931), 148--152.

\bibitem[GP22]{Gla}
A.~Glazyrin and I.~Pak, Domes over curves, \emph{International Mathematics Research Notices} \textbf{2022} (2022), 14067-14104.

\bibitem[KM96]{KM}
M.~Kapovich and J.Millson, The symplectic geometry of polygons in Euclidean space, \emph{J. Differ. Geom.}~\textbf{44} (1996), 479–513.

\bibitem[Ken05]{Ken}
R.~Kenyon, \emph{Open problems}, Personal website, \href{https://gauss.math.yale.edu/~rwk25/openprobs/}{https://gauss.math.yale.edu/\/\~{}rwk25/openprobs/index.html} (version: 2024-06-09).

\bibitem[Kly96]{Kly}
A.~Klyachko, Spatial polygons and stable configurations of points in the projective line, in \emph{Algebraic geometry and its applications, Aspects Math.}, 1996, 67-84.

\bibitem[Sta12]{Sta}
R.~Stanley, {\em Enumerative combinatorics}, vol.~1 (second ed.),
Cambridge Univ.~Press, 2012, 626~pp.

\bibitem[Stei57]{Stei}
H.~Steinhaus,
Probl\`{e}me 175,
\emph{Colloq. Math.}~\textbf{4} (1957), 243.

\bibitem[Stew19]{Stew}
I.~Stewart,
Tetrahedral chains and a curious semigroup,
\emph{Extracta Math.}~\textbf{34} (2019), 99--122.

\bibitem[\'{S}wi59]{Swi}
S.~\'{S}wierczkowski,
On chains of regular tetrahedra,
\emph{Colloq. Math.}~\textbf{7} (1959), 9-10.

\end{thebibliography}
\end{document}